\documentclass[12pt,a4paper,psamsfonts]{amsart}
\usepackage{amssymb,amscd,amsxtra,calc}
\usepackage{cmmib57}

\theoremstyle{plain}
    \newtheorem{thm}{Theorem}[section]
    
    \newtheorem{claim}[thm]{Claim}
    
    \newtheorem{lemma}[thm]{Lemma}
    \newtheorem{proposition}[thm]{Proposition}
    \newtheorem{question}[thm]{Question}
    \newtheorem{theorem}[thm]{Theorem}

\theoremstyle{definition}

    \newtheorem{remark}[thm]{Remark}
\theoremstyle{remark}

    \newtheorem{setup}[thm]{}

\newcommand{\BCC}{\mathbb{C}}
\newcommand{\C}{\mathbb{C}}

\newcommand{\Q}{\mathbb{Q}}

\newcommand{\R}{\mathbb{R}}
\newcommand{\BZZ}{\mathbb{Z}}
\newcommand{\Z}{\mathbb{Z}}

\newcommand{\OO}{\mathcal{O}}

\newcommand{\alb}{\operatorname{alb}}
\newcommand{\Amp}{\operatorname{Amp}}
\newcommand{\Aut}{\operatorname{Aut}}

\newcommand{\Gal}{\operatorname{Gal}}
\newcommand{\GL}{\operatorname{GL}}

\newcommand{\id}{\operatorname{id}}
\newcommand{\Imm}{\operatorname{Im}}
\newcommand{\Ker}{\operatorname{Ker}}
\newcommand{\NE}{\overline{\operatorname{NE}}}
\newcommand{\Nef}{\operatorname{Nef}}
\newcommand{\NS}{\operatorname{NS}}

\newcommand{\Sing}{\operatorname{Sing}}

\newcommand{\Alb}{\operatorname{Alb}}

\begin{document}

\title[Compact K\"ahler manifolds with automorphism groups of maximal rank]{
Compact K\"ahler manifolds with automorphism groups of maximal rank}

\author{De-Qi Zhang}
\address
{
\textsc{Department of Mathematics} \endgraf
\textsc{National University of Singapore, 10 Lower Kent Ridge Road,
Singapore 119076
}}
\email{matzdq@nus.edu.sg}

\begin{abstract}
For an automorphism group $G$ on an $n$-dimensional ($n \ge 3$) normal projective variety or
a compact K\"ahler manifold $X$ so that
$G$ modulo its subgroup $N(G)$ of null entropy elements is an abelian group of maximal rank $n-1$,
we show that
$N(G)$ is virtually contained in $\Aut_0(X)$, the $X$ is a quotient of a complex torus $T$
and $G$ is mostly descended from the symmetries on the torus $T$, provided that
both $X$ and the pair $(X, G)$ are minimal.
\end{abstract}


\thanks{The author is supported by an ARF of NUS}

\maketitle

\section{Introduction}

We work over the field $\BCC$ of complex numbers.
For a linear transformation $L$ on a finite-dimensional vector space $V$
over $\C$ or its subfields,
its {\it spectral radius} is defined as
$$\rho(L) := \max \{|\lambda| \, ; \, \lambda \, \in \, \C \,\,\, \text{is an eigenvalue of} \,\,\, L\}.$$

Let $X$ be a compact complex K\"ahler manifold and $Y$ a normal projective variety,
and let $g \in \Aut(X)$ and $f \in \Aut(Y)$.
Define the (topological) {\it entropy} $h(*)$ and {\it first dynamical degrees} $d_1(*)$ as:

$$\begin{aligned}
h(g) : &= \log \rho(g^* \, | \, \oplus_{i \ge 0} \, H^i(X, \C)), \\
d_1(g) : &= \rho(g^* \, | \, H^2(X, \C)) \,\, (= \rho(g^* \, | \, H^{1,1}(X))), \\
d_1(f) : &= \rho(f^* \, | \, \NS_{\C}(Y))
\end{aligned}$$
where $\NS_{\C}(Y) := \NS(Y) \otimes_{\Z} \C$ is the {\it complexified Neron-Severi group}.
By the fundamental work of Gromov and Yomdin, the above definition of entropy is equivalent to its original definition
(cf. \cite[\S 2.2]{DS} and the references therein).
Further, when $Y$ is smooth, the above two definitions of $d_1(*)$ coincide;
for $\Q$-factorial $Y$ (cf. \cite[0.4(1)]{KM}), we have $d_1(f) = d_1(\widetilde{f})$
where $\widetilde{f}$ is the lifting of $f$ to the one on an $\Aut(Y)$-equivariant resolution of $Y$.
We call $\tau := g$ or $f$, of {\it positive entropy} (resp. {\it null entropy}) if
$d_1(\tau) > 1$ (resp. $d_1(\tau) = 1$), or equivalently
$h(\tau) > 0$ (resp. $h(\tau) = 0$) in the case of compact K\"ahler manifold.

We say that the induced action $G \, | \, H^{1,1}(X)$ is
{\it Z-connected} if its Zariski-closure in $\GL(H^{1,1}(X))$ is connected
with respect to the Zariski topology;
in this case, the {\it null set}
$$N(G) := \{g \in G \, | \, g \,\, \text{is of null entropy} \}$$
is a (necessarily normal) subgroup of $G$ (cf. \cite[Theorem 1.2]{Z-Tits}).
In \cite{Z-Tits}, we have proved:

\begin{theorem} (cf.~\cite{Z-Tits}) \label{Z-TitsTh}
Let $X$ be an $n$-dimensional $(n \ge 2)$
compact complex K\"ahler manifold and $G$ a subgroup of $\Aut(X)$.
Then one of the following two assertions holds:
\begin{itemize}
\item[(1)]
$G$ contains a subgroup isomorphic to the non-abelian free group $\BZZ * \BZZ$,
and hence $G$ contains subgroups isomorphic to non-abelian free groups of all countable ranks.
\item[(2)]
There is a finite-index subgroup $G_1$ of $G$ such that
the induced action $G_1 \, | \, H^{1,1}(X)$ is solvable and Z-connected.
Further, the subset
$$N(G_1) := \{g \in G_1 \,\, | \,\, g \,\,\, \text{\rm is of null entropy}\}$$
of $G_1$ is a normal subgroup of $G_1$ and the quotient group $G_1/N(G_1)$ is a
free abelian group of rank $r \le n-1$. We call this $r$ {\bf the rank of $G_1$ and denote it as $r = r(G_1)$}.
\end{itemize}
\end{theorem}

Therefore, we are interested in the group $G \le \Aut(X)$ where
$G \, | \, H^{1,1}(X)$ is solvable and $Z$-connected and
that the rank $r(G) = \dim X - 1$ (maximal value).
In the following, denote by $\Aut_0(X)$ the {\it identity connected component} of $\Aut(X)$.
A group {\it virtually has a property} (P) if a finite-index subgroup of it has the property (P).

A complex torus has lots of symmetries.
Conversely, our Theorem \ref{ThA} of \cite{NullG}) (see also Theorem \ref{ThC} for non-algebraic manifolds)
says that the maximality $r(G) = \dim X - 1$
occurs only when $X$ is a quotient of a complex torus $T$
and $G$ is mostly descended from the symmetries on the torus $T$.

The statement of Theorem \ref{ThA} involves
minimal varieties and canonical singularities, but our
method uses only well known and precisely referred facts in Algebraic Geometry
rather than the technical part of the Minimal Model Program,
and hence is accessible.

Recall that in \cite{Z-Tits}, we constructed a quasi-nef sequence
$$L_1 \cdots L_k \, \in \, \overline{(L_1 \cdots L_{k-1}) \cdot \Nef(X)} \, \subset \, H^{k,k}(X, \R)$$
such that
$$g^*(L_1 \cdots L_k) = \chi_1(g) \cdots \chi_k(g) \, (L_1 \cdots L_k)$$
with characters $\chi_i : G \to (\R_{> 0} , \times)$
and a homomorphism
$$\varphi : G \, \to \, (\R^{\oplus n-1}, +), \hskip 1pc g \, \mapsto \, (\log \chi_1(g), \dots, \log \chi_{n-1}(g))$$
with
$\Ker \varphi = N(G)$,
and the image $\varphi(G)$ discrete (and hence a lattice) in $\R^{\oplus n-1}$.

Consider the following:

\par \vskip 1pc \noindent
{\bf Hypothesis (C)}: The discrete image (a lattice) of every
quasi-nef sequence induced injective homomorphism
$\overline{\varphi} : G/N(G) \rightarrow (\R^{\oplus n-1}, +)$ is, up to finite-index, the standard lattice.
Namely, $G/N(G)$ is freely generated by cosets $g_i N(G)$ so that
$\varphi(g_i)$ equals the $i$-th coordinate $(0, \cdots, 0, \log \chi_i(g_i), 0, \cdots, 0)$.

\par \vskip 1pc


\begin{theorem} 
\label{ThA}
Let $X$ be an $n$-dimensional $(n \ge 3)$
normal projective variety and $G \le \Aut(X)$ a subgroup such that
the induced action $G \, | \, \NS_{\C}(X)$ is solvable and $Z$-connected and that
the rank $r(G) = n - 1$ $($i.e., $G/N(G) = \Z^{\oplus n-1})$.
Assume the
following
conditions:

\begin{itemize}
\item[(i)]
$X$ has at worst canonical, quotient singularities {\rm (cf. \cite[Definition 2.34]{KM})}.
\item[(ii)]
$X$ is a minimal variety, i.e., the canonical divisor $K_X$ is nef {\rm (cf. \cite[0.4(3)]{KM})}.
\item[(iii)]
The pair $(X, G)$ is minimal in the sense of $\ref{conv}$.
\item[(iv)] Hypothesis (C). See \ref{proj} for more details.
\end{itemize}
Then the following four assertions hold.
\begin{itemize}
\item[(1)]
The induced action $N(G) \, | \, \NS_{\C}(X)$ is a finite group.
\item[(2)]
$G \, | \, \NS_{\C}(X)$ is a virtually free abelian group of rank $n-1$.
\item[(3)]
Either $N(G)$ is a finite subgroup of $G$ and hence
$G$ is a virtually free abelian group of rank $n-1$, or
$X$ is an abelian variety and the group
$N(G) \cap \Aut_0(X)$ has finite-index in $N(G)$ and is Zariski-dense in $\Aut_0(X)$ $(\cong X)$.
\item[(4)]
We have $X \cong T/F$ for a finite group $F$ acting freely outside a finite set of an abelian variety $T$.
Further,
for some finite-index subgroup $G_1$ of $G$, the action of $G_1$ on $X$ lifts to
an action of $G_1$ on $T$.
\end{itemize}
\end{theorem}

In \cite{max}, we assumed that (i) $G$ is abelian and (ii) the absence
of point wise $G$-fixed subvarieties of positive dimension or $G$-periodic rational curves or $Q$-tori.
In the current paper, these two restrictions are replaced by the natural minimality condition on $X$
and the pair $(X, G)$, and that $G \, | \, \NS_{\C}(X)$ is solvable, the latter of which is natural
in view of Theorem \ref{Z-TitsTh}. The quotient singularities assumption in Theorem \ref{ThA}
is necessary because an effective characterization of torus quotient is only available in dimension three
by \cite{SW} where the bulk of the argument is to show that the variety has only quotient singularities.

The lack of the abelian-ness assumption on $G$ makes our argument much harder, for instance we cannot
simultaneously diagonalize $G \, | \, \NS_{\C}(X)$ or find enough number of linearly independent
common nef eigenvectors of $G$ as required in \cite{DS} for abelian groups.

Theorems \ref{ThA} and \ref{ThC} answer \cite[Question 2.17]{Z-Tits}, assuming the conditions here.
When $G$ is abelian, the finiteness of $N(G)$ is proved in the inspiring paper
of Dinh-Sibony \cite[Theorem 1]{DS} (cf. also \cite{CY3}),
assuming only $r(G) = n-1$.
For non-abelian $G$, the finiteness of $N(G)$ is not true and
we can at best expect
that $N(G)$ is virtually included in $\Aut_0(X)$ (as done in Theorems \ref{ThA} and \ref{ThC}),
since a larger group $\widetilde{G} := \Aut_0(X) \, G$ satisfies
$$\widetilde{G} \, | \, \NS_{\C}(X) = G \, | \, \NS_{\C}(X), \,\,
N(\widetilde{G}) = \Aut_0(X) . N(G) \ge \Aut_0(X), \,\,
\widetilde{G}/N(\widetilde{G}) \cong G/N(G).$$

There are examples $(X, G)$ with rank $r(G) = \dim X - 1$ and $X$ complex tori or their quotients
(cf. \cite[Example 4.5]{DS}, \cite[Example 1.7]{CY3}).

The proof of Theorem \ref{ThA} is much harder than that of Theorem \ref{ThC}
because of the presence of singularities on $X$.

The conditions (i) - (iii) in Theorem \ref{ThA} are quite necessary in deducing $X \cong T/F$ as in Theorem \ref{ThA}(4).
Indeed, if $X \cong T/F$ as in Theorem \ref{ThA}(4), then $X$ has only quotient singularities
and $d K_X \sim 0$ (linear equivalence) with $d = |F|$, and we may even assume that $X$ has only canonical singularities
if we replace $X$ by its global index-$1$ cover; thus $X$ is a minimal variety.
If the pair $(X, G)$ is not minimal so that there
is a non-isomorphic $G_1$-equivariant birational morphism
$X \to Y$ as in \ref{conv}, then the exceptional locus of this morphism
is $G_1$- and hence $G$-periodic, contradicting the fact that the rank $r(G) = n-1$ (cf. the proof of Claim \ref{per}).

The {\it first key step} in proving Theorem \ref{ThA} is the
analysis of our quasi-nef sequence $L_1 \dots L_k$ ($0 < k < n$)
(cf. \cite[\S 2.2]{Z-Tits}) and we are able to show that $L_i$ can actually
be taken to be nef, when the rank $r(G) = n-1$ (cf. Lemma \ref{newL}).
The {\it second key step} is Theorem \ref{ThB}
where we split $G$ as $N(G) \, H$
such that $H \, | \, H^{1,1}(X)$ is free abelian; thus we have a nef and big class $A$
as the sum of nef common eigenvectors of $H$,
leading to the vanishing of $A^{n-i} . c_i(X)$,
where $c_i(X)$ ($i = 1, 2$) are Chern classes (cf. \cite[pages 265-267]{SW}).
The {\it third key step} is to use the minimality of $(X, G)$ and Kawamata's base point freeness
result for $\R$-divisor (cf. \cite[Theorem 3.9.1]{BCHM}, available only for projective
variety at the moment) to deduce the vanishing
of $c_2(X)$ as a linear form; this vanishing does not directly follow from the vanishing of $A^{n-2} . c_2(X)$
because $A$ may not be ample. Now Theorem \ref{ThA}(4) follows from the vanishing of $c_i(X)$ ($i = 1, 2$)
and the characterization of torus quotient originally deduced from Yau's deep result
(cf. \cite{Be}).

\begin{remark}\label{rThA}
$ $
\begin{itemize}
\item[(1)]
When $\dim X = 3$, the (i) in Theorem \ref{ThA} can be replaced by:
(i)' $X$ has at worst canonical singularities
(cf. Proof of Lemma \ref{ampleA} and \cite[Corollary at p. 266]{SW}).
\item[(2)]
Theorems \ref{ThA} and \ref{ThC} are not true when $n := \dim X = 2$.
We used $n \ge 3$ to deduce the vanishing of $c_2(X) . A^{n-2}$
as commented above.
\end{itemize}
\end{remark}

With Theorem \ref{ThB} in mind, we ask:

\begin{question}
Suppose a group $G$ acts on a compact complex K\"ahler manifold $($say a complex torus$)$ such that the null set
$N(G)$ is a subgroup of $($and hence normal in$)$ $G$
and the quotient $G/N(G)$ is a free abelian group.
Under what condition, can we write $G$ $($or its finite-index subgroup$)$
as $G = N(G) \rtimes H$ with $H \le G$ a free abelian subgroup of $G$?
\end{question}

\noindent
{\bf Acknowledgement.}
I would like to thank Frederic Campana for pointing out the algebraicity
of the Calabi-Yau factors in the Bogomolov decomposition.

\section{Proof of Theorems}

In this section, we prove Theorem \ref{ThA} in the introduction
and the three results below.

When the $X$ below is non-algebraic, we don't require the minimality of the pair $(X, G)$.

\begin{theorem}\label{ThC}
Let $X$ be an $n$-dimensional $(n \ge 3)$
compact complex K\"ahler manifold which is not algebraic.
Let $G \le \Aut(X)$ be a subgroup such that
the induced action $G \, | \, H^{1,1}(X) \le \Aut(H^{1,1}(X))$ is solvable and $Z$-connected and
that the rank $r(G) = n - 1$ $($i.e., $G/N(G) = \Z^{\oplus n-1})$.
Assume that $X$ is minimal, i.e., the canonical divisor $K_X$ is
contained in the closure of the K\"ahler cone of $X$.
Assume also Hypothesis (C).
Then the following four assertions hold.
\begin{itemize}
\item[(1)]
The induced action $N(G) \, | \, H^{1,1}(X)$ is a finite group.
\item[(2)]
$G \, | \, H^{1,1}(X)$ is a virtually free abelian group of rank $n-1$.
\item[(3)]
Either $N(G)$ is a finite subgroup of $G$ and hence
$G$ is a virtually free abelian group of rank $n-1$, or
$X$ is a complex torus and the group
$N(G) \cap \Aut_0(X)$ has finite-index in $N(G)$ and is Zariski-dense in $\Aut_0(X)$ $(\cong X)$.
\item[(4)]
We have $X \cong T/F$ for a finite group $F$ acting freely outside a finite set of a complex torus $T$.
Further,
for some finite-index subgroup $G_1$ of $G$, the action of $G_1$ on $X$ lifts to
an action of $G_1$ on $T$.
\end{itemize}
\end{theorem}

\par \vskip 1pc
The $X$ or the pair $(X, G)$ below is not assumed to be minimal.

\begin{theorem}\label{ThB}
Let $X$ be an $n$-dimensional $(n \ge 2)$
compact complex K\"ahler manifold and $G \le \Aut(X)$ a subgroup such that
the induced action $G \, | \, H^{1,1}(X) \le \Aut(H^{1,1}(X))$ is solvable and $Z$-connected and
that the rank $r(G) = n - 1$ $($i.e., $G/N(G) = \Z^{\oplus n-1})$.
Assume also Hypothesis (C).
Then, replacing $G$ by its finite-index subgroup,
we can find a subgroup $H \le G$ such that:

\begin{itemize}
\item[(1)]
$G = N(G) \, H$;
\item[(2)]
$G \, | \, H^{1,1}(X) = (N(G) \, | \, H^{1,1}(X)) \rtimes (H \, | \, H^{1,1}(X))$;
\item[(3)]
$N(G) \, | \, H^{1,1}(X)$ is unipotent; and
\item[(4)]
The induced action $H \, | \, H^{1,1}(X) \le \Aut(H^{1,1}(X))$
is a free abelian group of rank $n-1$, i.e., $H \, | \, H^{1,1}(X) \cong \Z^{\oplus n-1}$.
\end{itemize}
\end{theorem}

In the process of proving Theorem \ref{ThB}, we also deduce:

\begin{proposition}\label{primelt}
For the $X$ and $G$ in Theorem $\ref{ThB}$, replacing $G$ by its finite-index subgroup,
we can find some $g_0 \in G \setminus N(G)$, such that
the first dynamical degrees satisfy:
$$d_1(g) = d_1(g_0)^t$$
for every $g \in G$ with $t \in \Z_{\ge 0}$ depending on $g$.
\end{proposition}

\begin{remark}\label{rThB}
(1) In Theorem \ref{ThB} and Proposition \ref{primelt},
if $X$ is a normal projective variety, then a similar proof implies the same conclusions but with
all the $H^{1,1}(X)$ in Theorem \ref{ThB} replaced by $\NS_{\C}(X)$.

(2) To allow singularities of $X$ in Theorem \ref{ThC},
we need the K\"ahler version of the birational contraction theorem
\cite[Theorem 3.9.1]{BCHM} and
Miyaoka's pseudo-effectivity of $c_2(X)$ for minimal variety $X$,
both of which seem to be very hard to confirm, since the Minimal Model Program
has not been fully developed for K\"ahler manifolds.
\end{remark}

\begin{setup}{\bf $s$-cycles and minimal pairs}\label{conv}

Let $X$ be an $n$-dimensional compact K\"ahler manifold or a normal projective variety.
When $X$ is projective, set $\NS_{\R}(X) := \NS(X) \otimes_{\Z} \R$.
A codimension-$s$ (i.e., dimension-$(n-s)$) {\it cycle} $D$ is an element in
$H^{s,s}(X, \R) := H^{s, s}(X) \cap H^{2s}(X, \R)$
(resp. a linear combination of $(n-s)$-dimensional subvarieties with coefficients in $\R$)
when $X$ is K\"ahler (resp. projective).
Two codimension-$s$ cycles $D_i$ are {\it numerically equivalent}, denoted as $D_1 \equiv D_2$,
if $(D_1 - D_2) . L_1 \dots L_{n-s} = 0$ for all $L_i$ in $H^{1,1}(X, \R)$
(resp. in $\NS_{\R}(X)$), where we use $D . L$ to denote the cup product (resp. intersection)
for K\"ahler (resp. projective) $X$.
Denote by $[D]$ the {\it numerical equivalence class} containing $D$
and
$$N^s(X) := \{[D] \, ; \, D \,\, \text{is a codimension-s cycle} \}$$
which is a finite-dimensional $\R$-vector space.
We will loosely write $D \in N^s(X)$ by abuse of notation.
Note that $N^1(X) = \NS_{\R}(X)$ when $X$ is projective.
Denote by $K(X)$ (resp. $\Amp(X)$) the open K\"ahler (resp. ample) cone
and $\overline{K(X)}$ (resp. $\Nef(X)$) its closure in $H^{1,1}(X, \R)$ (resp. in $\NS_{\R}(X)$).
Elements in $\overline{K(X)}$ and $\Nef(X)$ are called {\it nef}.

Let $X$ be a normal projective variety with at worst canonical singularities (cf. \cite[Definition 2.34]{KM})
and $G \le \Aut(X)$ a subgroup such that the null set $N(G)$ is a subgroup (and hence normal in $G$).
The pair $(X, G)$ is {\it non-minimal} if:
there are a finite-index subgroup $G_1$ of $G$ and a {\it non-isomorphic} $G_1$-equivariant birational morphism $X \to Y$
onto a normal projective variety $Y$ with at worst {\it isolated} canonical singularities. The pair
$(X, G)$ is {\it minimal} if it is not non-minimal.

\end{setup}

\begin{lemma}\label{modfin}
Let $G$ be a group, $H \lhd G$ a finite normal subgroup.
Suppose that
$$G/H = \langle \bar g_1 \rangle \times \cdots \times \langle \bar g_r \rangle
\cong \Z^{\oplus r}$$
for some $r \ge 1$ and $g_i \in G$.
Then there is an integer $s > 0$ such that
$G_1 := \langle g_1^s, \dots, g_r^s \rangle$ satisfies
$$G_1 = \langle g_1^s \rangle \times \cdots \times \langle g_r^s \rangle
\cong \Z^{\oplus r}$$
and it is a finite-index subgroup of $G$; further,
the quotient map $\gamma : G \to G/H$ restricts to an
isomorphism $\gamma \, | \, G_1 : G_1 \to \gamma(G_1)$ onto a finite-index subgroup of $G/H$.
\end{lemma}

\begin{proof}
We only need to find $s > 0$ such that $g_i^s$ and $g_j^s$ are commutative to each other for all $i, j$.
Since $G/H$ is abelian, the commutator subgroup $[G, G] \le H$.
Thus the commutators $[g_1^t, g_2]$ ($t > 0$) all belong to $H$.
The finiteness of $H$ implies that $[g_1^{t_1}, g_2] = [g_1^{t_2}, g_2]$ for some $t_2 > t_1$,
which implies that $g_1^{s_{12}}$ commutes with $g_2$, where $s_{12} := t_2 - t_1$.
Similarly, we can find an integer $s_{1j} > 0$ such that $g_1^{s_{1j}}$ commutes with $g_j$.
Set $s_1 := s_{12} \times \cdots \times s_{1r} > 0$. Then $g_1^{s_1}$ commutes with
every $g_j$.
Similarly, for each $i$,
we can find an integer $s_{i} > 0$ such that $g_i^{s_i}$ commutes with $g_j$ for all $j$.
Now $s := s_1 \times \cdots \times s_{r} > 0$ will do the job. This proves the lemma.
\end{proof}

{\it From now on till $\ref{Pfprimelt}$, we prove Theorem $\ref{ThB}$ and Proposition $\ref{primelt}$.}

\begin{setup}\label{phi}
By \cite[Proof of Theorem 1.2, \S 2.2]{Z-Tits}, there
is a quasi-nef sequence $L_1 \dots L_{k} \in H^{k,k}(X)$ ($1 \le k \le n-1$) which is nonzero
in $N^k(X)$, such that
$$g^*(L_1 \dots L_k) = \chi_1(g) \cdots \chi_k(g) (L_1 \dots L_k)$$
for all $k = 1, \dots, n-1$ with characters $\chi_i : G \to (\R_{> 0} , \times)$, and that
$$\begin{aligned}
\varphi : G &\to (\R^{n-1}, +) \\
g &\mapsto (\log \chi_1(g), \dots, \log \chi_{n-1}(g))
\end{aligned}$$
is a homomorphism having
$$\Ker(\varphi) = N(G)$$
and the following subgroup of $\R^{\oplus n-1}$ as its image (with $r(G) = n- 1$ now)
which is discrete and hence a lattice in $\R^{\oplus n-1}$:
$$\Imm(\varphi) = \Z^{\oplus r(G)} .$$
\end{setup}



From now on till the end of the paper, we assume the following:

\begin{setup}\label{proj}
{\bf Hypothesis C}: The discrete image (a lattice) of every
quasi-nef sequence induced injective homomorphism
$\overline{\varphi} : G/N(G) \rightarrow (\R^{\oplus n-1}, +)$ is, up to finite-index, the standard lattice.
Namely, $G/N(G)$ is freely generated by cosets $g_i N(G)$ so that
$\varphi(g_i)$ equals the $i$-th coordinate $(0, \cdots, 0, \log \chi_i(g_i), 0, \cdots, 0)$.
\end{setup}

So $\Imm(\varphi)$ is the direct product of $n-1$ cyclic groups $\Imm (\log \chi_i)$
which are the $\varphi$-images of $\langle g_i \rangle$ say, with
$$\lambda_i := \chi_i(g_i) > 1, \,\,\,\, \chi_j(g_i) = 1 \,\, (j \ne i).$$
Hence
$$G/N(G) = \langle \overline{g_1}, \dots, \overline{g_{n-1}} \rangle
= \langle \overline{g_1} \rangle \oplus \cdots \langle \overline{g_{n-1}} \rangle \cong \Z^{\oplus n-1}$$
and
$$\Imm(\varphi) = \langle \varphi(g_1), \dots, \varphi(g_{n-1}) \rangle
= \langle \varphi(g_1) \rangle \oplus \cdots \langle \varphi(g_{n-1}) \rangle \cong \Z^{\oplus n-1}.$$

By the generalization of Perron-Frobenius theorem \cite{Bi} applied to the action of $g_i^{\pm 1}$
on the closure $\overline{K(X)}$ of the K\"ahler cone $K(X)$
(which spans $H^{1,1}(X, \R)$), there are nonzero nef $L_{g_i^{\pm 1}} \in \overline{K(X)}$
(which we fix now)
such that
$$(g_i^{\pm 1})^* L_{g_i^{\pm 1}} = d_1(g_i^{\pm 1}) L_{g_i^{\pm 1}} .$$
Note that the first dynamical degree
$d_1(g_i^{\pm 1}) > 1$ since $g_i^{\pm 1} \not\in N(G)$.

\begin{lemma}\label{newL}
Let $k \in \{1, \dots, n-1\}$.
Let $M \in \overline{K(X)}$ be a nonzero nef element such that $g_k^*M = \lambda M$
for some $\lambda \ne 1$ $($e.g., we can take $M = L_{g_k}, L_{g_k^{-1}})$.
\begin{itemize}
\item[(0)]
For $D$ $($e.g., $D = L_1 \dots L_k)$ in $\overline{P^k(X)}$
$($the closure of K\"ahler $(k, k)$-forms as defined in \cite[before Lemma A.3]{NZ}$)$, $D = 0$ in $N^k(X)$
if and only if $D = 0$ in $H^{k,k}(X)$.

\item[(1)]
Suppose that $\lambda \ne 1/\lambda_k$ $($e.g., we can take $M = L_{g_k})$.
Then $L_1 \dots L_k . M = 0$ in $H^{k+1}(X)$.

\item[(2)]
Suppose that $k \ge 2$. Then
$L_1 \dots L_{k-1} . M$ is nonzero in $N^k(X)$.

\item[(3)]
Suppose that $k = 1$ and $\lambda > 1$ $($e.g., $M = L_{g_1})$. Then $M$ is parallel to $L_1$ in $H^{1,1}(X)$
and hence $\lambda = \lambda_1$.
In particular, $L_{g_1}$ is parallel to $L_1$ in $H^{1,1}(X)$.

\item[(4)]
Suppose that $k \ge 2$ and $\lambda \ne 1/\lambda_k$. Then
$L_1 \dots L_k$ and $L_1 \dots L_{k-1} . M$ are parallel in $N^{k}(X)$.
In particular, $\lambda = \lambda_k$.

\item[(5)]
For all $1 \le s \le n-1$,
$L_1 \dots L_s$ and $L_{g_1} \dots L_{g_s}$ are parallel in $N^{s}(X)$.

\item[(6)]
$\lambda$ equals either $\lambda_k > 1$, or $1/\lambda_k < 1$.

\item[(7)]
For all $1 \le s \le n-1$, we have
$$d_1(g_s) = \lambda_s \, (= \chi_s(g_s)), \,\,\,\,
d_1(g_s^{-1}) = d_1(g_s) \, (> 1).$$

\item[(8)]
For all $1 \le s \le n-1$, $L_1 \dots L_{n-1} . L_{g_s^{-1}}$
is a nonzero $($positive$)$ scalar in $N^n(X) = H^{n,n}(X, \R) \cong \R$.

\end{itemize}
\end{lemma}

\begin{proof}
The assertion (0) is well known; see e.g. \cite[Lemma A.4]{NZ}.

(1) Suppose the contrary that $L_1 \dots L_k . M$ is nonzero in $H^{k,k}(X)$ (i.e., in $N^{k+1}(X)$ by
the assertion (0)).
Since
$$g_k^*(L_1 \dots L_k . L_{k+1}) = \lambda_k (L_1 \dots L_k . L_{k+1}), \,\,\,\,
g_k^*(L_1 \dots L_k . M) = \lambda_k \lambda (L_1 \dots L_k . M)$$
with $\lambda_k \ne \lambda_k \lambda$, we have
$L_1 \dots L_{k+1} . L_{g_k} \ne 0$ in $N^{k+1}(X)$ (cf. \cite[Lemma 2.3]{Z-Tits}, \cite[Lemma 4.4]{DS}).
Inductively, the same $g_k^*$ action and [ibid.] imply that
$L_1 \dots L_t . M$ is nonzero in $N^{t+1}(X)$
for all $k \le t \le n-1$.
In particular, $L_1 \dots L_{n-1} . M$
equals a positive scalar $b_n$ in $N^n(X) \cong \R$.
Thus $b_n = g_k^* b_n = g_k^*(L_1 \dots L_{n-1} . M) = \lambda_k \lambda \, b_n \ne b_n$,
by the condition on $\lambda$. This is a contradiction.

(2) Since $g_k^*M = \lambda M$ and $g_k^* L_1 = L_1$
with $\lambda \ne 1$, we have $L_1 . M \ne 0$ in $N^2(X)$ by \cite[Lemma 4.4]{DS}.

Let $k \ge 3$ and let $t \ge 1$ be the largest integer such that
$L_1 \dots L_t . M \ne 0$ in $N^{t+1}(X)$.
If $t \ge k-1$ then (2) is true.
Suppose the contrary that $t \le k-2$.
Since
$$g_k^*(L_1 \dots L_t . L_{t+1}) = L_1 \dots L_t . L_{t+1}, \,\,\,\,
g_k^*(L_1 \dots L_t . M) = \lambda (L_1 \dots L_t . M)$$
with $\lambda \ne 1$, we have $L_1 \dots L_t . L_{t+1} . M \ne 0$ in $N^{t+2}(X)$
(cf. \cite[Lemma 2.3]{Z-Tits}, \cite[Lemma 4.4]{DS}). This contradicts the maximality of $t$.
So $t \ge k-1$. Hence (2) is true.

(3) By (1), $L_1 . L_{g_1} = 0$ in $H^{1,1}(X)$,
so $L_1$ and $M$
are parallel in $H^{1,1}(X)$ (cf. \cite[Corollary 3.2]{DS}).

(4) We have $L_1 \dots L_k . M = 0$ by (1).
Then (4) follows
(cf. \cite[Lemma 2.3]{Z-Tits}, \cite[Corollary 3.5]{DS}).
Indeed, for the last part, just apply $g_k^*$ to the equality of the first part (cf. (2)).

(5) follows from (3) and (4) with $M = L_{g_k}$, by induction on $k$.

(6) By (2), $L_1 \dots L_{k-1} . M \ne 0$ in $N^k(X)$.
Suppose the contrary that $\lambda \ne \lambda_k^{\pm}$.
Since
$$g_k^*(L_1 \dots L_{k-1} . L_{k}) = \lambda_k (L_1 \dots L_k), \,\,\,\,
g_k^*(L_1 \dots L_{k-1} . M) = \lambda (L_1 \dots L_{k-1} . M)$$
with $\lambda_k \ne \lambda$, we have
$L_1 \dots L_k . M \ne 0$ in $N^{k+1}(X)$ (cf. \cite[Lemma 2.3]{Z-Tits}, \cite[Lemma 4.4]{DS}), contradicting (1).

(7) follows from (6) with $M = L_{g_k}$, $L_{g_k^{-1}}$.

(8) Set $M := L_{g_k^{-1}}$. By (2), we have the non-vanishing of
$L_1 \dots L_{k-1} . M$ in $N^k(X)$. As in the proof of (1), inductively, the action of $g_k^*$
implies the non-vanishing of $L_1 \dots L_t . M$ in $N^{t+1}(X)$ for all $k - 1 \le t \le n-1$.
\end{proof}

By Lemma \ref{newL}, we may and will take
$$L_i = L_{g_i} .$$

\begin{lemma}\label{Lsq}
For all $1 \le k \le n-1$, as elements in $H^{2,2}(X)$, we have
$$L_{g_k}^2 = 0, \,\,\, L_{g_k^{-1}}^2 = 0 .$$
In particular, in $H^{1,1}(X)$, when $g_k^*M = \lambda M$ for some nonzero nef
$M$, then $\lambda = d_1(g_k)$ and $M$ is parallel to $L_{g_k}$ if $\lambda > 1$, and
$\lambda^{-1} = d_1(g_k) = d_1(g_k^{-1})$ and $M$ is parallel to $L_{g_k^{-1}}$ if $\lambda < 1$;
so the choice of nef $L_{g_k}$ $($resp. $L_{g_k^{-1}})$ is unique, up to a positive multiple.
\end{lemma}

\begin{proof}
We take $L_i = L_{g_i}$.
Set $M = L_{g_k}$ or $L_{g_k^{-1}}$.

We Claim that
$L_1 \dots L_{k-1} . M^2$ vanishes in $N^{k+1}(X)$
(i.e., in $H^{k+1,k+1}(X)$, by Lemma \ref{newL}(0)).
When $M = L_{g_k}$, this is true by Lemma \ref{newL}(1).
When $M = L_{g_k^{-1}}$, suppose the contrary that
$L_1 \dots L_{k-1} . M^2$ is non-vanishing in $N^{k+1}(X)$.
Note that $L_1 \dots L_k . M \ne 0$ in $N^{k+1}(X)$ by Lemma \ref{newL}(8).
Since
$$\begin{aligned}
g_k^*(L_1 \dots L_{k-1} . M . M) &= (1/\lambda_k)^2 (L_1 \dots L_{k-1} . M . M), \\
g_k^*(L_1 \dots L_{k-1} . M . L_k) &= (1/\lambda_k) \lambda_k  (L_1 \dots L_{t-1} . L_k . L_t)
\end{aligned}$$
with $(1/\lambda_k)^2 \ne (1/\lambda_k) \lambda_k$,
we have the non-vanishing of $L_1 \dots L_k . M^2$ in $N^{k+2}(X)$
(cf. \cite[Lemma 2.3]{Z-Tits}, \cite[Lemma 4.4]{DS}).
Inductively, the action of $g_k^*$ and [ibid.] imply the
non-vanishing of $b_{t+2} := L_1 \dots L_t . M^2$ in $N^{t+2}(X)$, for all $k-1 \le t \le n-2$.
Thus $b_n = g_k^* b_n = \alpha (1/\lambda_k)^2 b_n \ne b_n$, where $\alpha = 1$ (when $k = n-1)$
or $\alpha = \lambda_k$ (when $k \le n-2$), a contradiction!
Hence the claim is true. In particular, the lemma is true when $k = 1$.

Let $k \ge 2$. Suppose the contrary that $M^2$ is nonzero in $H^{2,2}(X)$ (i.e., in $N^2(X)$ by Lemma \ref{newL}(0)).
By the claim above, we can choose $1 \le t \le k-1$ to be the smallest integer such that
$L_1 \dots L_t . M^2 = 0$ in $N^{t+2}(X)$.
Thus $L_1 \dots L_{t-1} . M^2 \ne 0$ in $N^{t+1}(X)$.
Note that $L_1 \dots L_t . M \ne 0$ in $N^{t+1}(X)$, by Lemma \ref{newL}(8)
(for $M = L_{g_k^{-1}}$) and by the non-vanishing of $L_1 \dots L_{n-1}$ in $N^{n-1}(X)$ (for $M = L_k$).
Since
$$ \begin{aligned}
g_k^*(L_1 \dots L_{t-1} . M . M) &= \lambda^2 (L_1 \dots L_{t-1} . M . M), \\
g_k^*(L_1 \dots L_{t-1} . M . L_t) &= \lambda  (L_1 \dots L_{t-1} . M . L_t)
\end{aligned}$$
with $\lambda^2 \ne \lambda$ (because $\lambda = \lambda_k$ or $\lambda_k^{-1}$
for $M = L_{g_k}$ or $L_{g_k^{-1}}$),
we have the non-vanishing of $L_1 \dots L_t . M^2$ in $N^{t+2}(X)$
(cf. \cite[Lemma 2.3]{Z-Tits}, \cite[Lemma 4.4]{DS}).
This contradicts the choice of $t$.

For the final part, we may suppose that $\lambda \ne 1$. By Lemma \ref{newL}, if $\lambda > 1$
(resp. $\lambda < 1$), we have $\lambda = d_1(g_k) = d_1(g_k^{-1})$
(resp. $\lambda^{-1} = d_1(g_k) = d_1(g_k^{-1})$).
Taking $M$ or $M + L_{g_k}$ (resp. $M$ or $M + L_{g_k^{-1}}$) as new $L_{g_k}$ (resp. $L_{g_k^{-1}}$),
the first part shows that, in $H^{2,2}(X)$, we have
$M^2 = L_{g_k}^2 = (M + L_{g_k})^2 = 0$ (resp. $M^2 = L_{g_k^{-1}}^2 = (M + L_{g_k^{-1}})^2 = 0$),
hence $M . L_{g_k} = 0$ (resp. $M . L_{g_k^{-1}} = 0$).
Then $M$ and $L_{g_k}$ (resp. $M$ and $L_{g_k^{-1}}$) are parallel in $H^{1,1}(X)$, by \cite[Corollary 3.2]{DS}.
\end{proof}

Set $L_n = L_{g_k^{-1}}$ for some $1 \le k \le n-1$. Then
$$A := L_1 + \dots + L_n$$
is a nef and big class
because $H^n \ge L_1 \dots L_{n-1} . L_n > 0$ by Lemma \ref{newL}(8).
We may also write
$g^*(L_1 \dots L_n) = \chi_1(g) \cdots \chi_n(g) (L_1 \dots L_n)$.
Then
$$\chi_1 \dots \chi_n = 1$$ since $L_1 \dots L_n$ is a nonzero (positive) scalar in $N^n(X) \cong \R$.

\begin{lemma}\label{centre}
$C := (Z(G) | H^{1,1}(X)) \cap (N(G) | H^{1,1}(X))$ is a finite subgroup of $\Aut(H^{1,1}(X))$.
\end{lemma}

\begin{proof}
Take any $z \, | \, H^{1,1}(X) \in C$.
Then $z \, | \, H^{1,1}(X) \in Z(G) \, | \, H^{1,1}(X)$ and hence $z \, | \, H^{1,1}(X)$ commutes with every $g_i \, | \, H^{1,1}(X)$.
Thus, $g_i^* z^*L_{g_i} = z^* g_i^* L_{g_i} = d_1(g_i) z^* L_{g_i}$, so
$z^*L_{g_i}$ and $L_{g_i}$ are parallel by Lemma \ref{Lsq}, and hence are equal, since
$z \in N(G)$ has the first dynamical degree $d_1(z) = 1$. By the same reasoning, $z^* L_{g_i^{-1}} = L_{g_i^{-1}}$.
Thus $z^*$ fixes the nef and big class $A = \sum_{i=1}^n L_i$ with $L_i = L_{g_i}$ and $L_n = L_{g_1^{-1}}$.
Hence $z^s \in \Aut_0(X)$
for some $s > 0$ by \cite[Proposition 2.2]{Li} (cf. \cite[Lemma 2.23]{JDG} and note
that a big class is the sum of a K\"ahler class and a positive real current, according to Demailly-Paun).
Thus $z^s$ acts trivially on the lattice $H^2(X, \Z)$ and hence also on $H^{1,1}(X)$.
So $C$ is a periodic group and defined over $\Z$, hence is a finite group
by Burnside's theorem (cf. \cite[Proof of Proposition 2.2]{Og}).
This proves the lemma.
\end{proof}

Now Theorem \ref{ThB} follows from the following:

\begin{proposition}\label{split}
Replacing $G$ by its finite-index subgroup and $g_i$ by some element in $g_i N(G)$,
the group $H := \langle g_1, \dots, g_{n-1} \rangle$ has its image $H \, | \, H^{1,1}(X)$ in $\Aut(H^{1,1}(X))$
a free abelian group of rank $n-1$
so that $G = N(G) \, H$,
$N(G) \, | \, H^{1,1}(X)$ is unipotent, and
$$G \, | \, H^{1,1}(X) = (N(G) \, | \, H^{1,1}(X)) \rtimes (H \, | \, H^{1,1}(X)) .$$
\end{proposition}

\begin{proof}
By \cite[Proposition 4.1]{Wo} and Selberg's lemma,
replacing $G$ by its finite-index subgroup, we may assume that $G \, | \, H^{1,1}(X)$ is torsion free.
Note that the set $U(G) := \{u \in G \, ; \, u \, | \, H^{1,1}(X)$ is unipotent$\}$
is a finite-index subgroup of $N(G)$ and a characteristic subgroup of $G$ (cf. \cite[Proof of Prop 2.2]{Og},
\cite[Theorem 3.1]{CWZ}). Replacing $G$ by a finite-index subgroup of $G_1 := U(G) \, \langle g_1, \dots, g_{n-1} \rangle$,
we may assume that $N(G) = U(G)$ (applying Lemma \ref{modfin} to the group $G_1/U(G)$).
Since $\Ker(G \to G \, | \, H^{1,1}(X)) \le N(G)$, we get
$$(*) \hskip 1pc (G \, | \, H^{1,1}(X))/(N(G) \, | \, H^{1,1}(X)) \cong G/N(G)
= \langle \overline{g_1}, \dots, \overline{g_{n-1}} \rangle \cong \Z^{\oplus n-1}$$
which is abelian and hence nilpotent.
By \cite[Thm 3, p.~48]{Se}, there is a subgroup $H$ of $G$ such that $H \, | \, H^{1,1}(X)$ is nilpotent
and $G = N(G) \, H$, after replacing $G$ by its finite-index subgroup.
Replacing $g_i$ by some element in $g_i N(G)$, we may assume that $g_i \in H$.

Suppose $H \, | \, H^{1,1}(X)$ is non-abelian. Then the commutator subgroup $[H, H] \, | \, H^{1,1}(X)$ is non-trivial
and is contained in $[G, G] \, | \, H^{1,1}(X)$ while the latter is contained in the group $N(G) \, | \, H^{1,1}(X)$
because $G/N(G)$ is abelian. Since $H \, | \, H^{1,1}(X)$ is nilpotent, its centre
($\le [H, H] \, | \,  H^{1,1}(X) \le N(G) \, | \, H^{1,1}(X)$)
contains a non-trivial element $z \, | \, H^{1,1}(X)$
(of course commutative with all $g_i \, | \, H^{1,1}(X)$).
The proof of Lemma \ref{centre} shows that $z \, | \, H^{1,1}(X) \in G \, | \, H^{1,1}(X)$ is of finite order,
contradicting the torsion freeness assumption of $G \, | \, H^{1,1}(X)$.

Thus, we may assume that $H \, | \, H^{1,1}(X)$ is abelian (and free because so is $G \, | \, H^{1,1}(X)$).
Replacing $H$ by $\langle g_1, \dots, g_{n-1} \rangle$, we may assume that
$H \, | \, H^{1,1}(X)$ is free abelian and of rank $n-1$ because it is generated by $n-1$ elements
and dominates $\Z^{\oplus n-1}$ via the {\it surjective} composite below (cf. the display (*) above):
$$\begin{aligned}
H \, | \, H^{1,1}(X) &\to (H \, | \, H^{1,1}(X)) / (H \, | \, H^{1,1}(X)) \cap (N(G) \, | \, H^{1,1}(X)) \\
&\to (G \, | \, H^{1,1}(X))/(N(G) \, | \, H^{1,1}(X)) \cong
\langle \overline{g_1}, \dots, \overline{g_{n-1}} \rangle \cong \Z^{\oplus n-1}.
\end{aligned}$$
Now the same dominance (between free abelian groups of the same rank)
implies that $(H \, | \, H^{1,1}(X)) \cap (N(G) \, | \, H^{1,1}(X)) = \{\id\}$
(or else $\Z^{\oplus n-1}$ is dominated by a free abelian group,
a quotient of $H$, of rank $\le n-2$, absurd!).
Thus the display of and hence the whole of the proposition follow.
\end{proof}

\begin{lemma}\label{Lg1}
$L_{g_s^{-1}}$ $(1 \le s \le n-1)$ are all parallel to $L_{g_1^{-1}}$ in $H^{1,1}(X)$.
\end{lemma}

\begin{proof}
Since $L_{g_i^{-1}}$ is unique up to a multiple (cf. Lemma \ref{Lsq}) and is parallel to
$L_{g_i^{-t}}$ for any $t > 0$, we may assume that
$G = N(G) \, H$ as in Proposition \ref{split}.
Since $g_i \, | H^{1,1}(X) \in H \, | \, H^{1,1}(X)$ ($1 \le i \le n-1$) are commutative to each other,
we may assume that both $L_{g_i}$ and $L_{g_i^{-1}}$ are common nef eigenvectors of $H$.
Set $L := L_{g_1^{-1}}$ and write $g_s^* L = \mu_s L$ with
$\mu_1 = 1/d_1(g_1^{-1}) = 1/d_1(g_1) = 1/\lambda_1$ (cf. Lemma \ref{newL}).
Applying $g_s^*$ to the nonzero scalar $L_{g_1} \dots L_{g_{n-1}} . L$ in $N^n(X) \cong \R$
in Lemma \ref{newL}, we get
$$L_{g_1} \dots L_{g_{n-1}} . L = (d_1(g_s) \, \mu_s) (L_{g_1} \dots L_{g_{n-1}} . L).$$
Hence $1 = d_1(g_s) \, \mu_s$ and $\mu_s = 1/d_1(g_s) = 1/d_1(g_s^{-1})$. Thus $(g_s^{-1})^*L = d_1(g_s^{-1}) L$
and the lemma follows (cf. Lemma \ref{Lsq}).
\end{proof}

\begin{setup}\label{Pfprimelt}
{\bf Proof of Proposition \ref{primelt}}
\end{setup}

Since $G \, | \, H^{1,1}(X)$ is solvable and $Z$-connected
and hence upper-triangularizable by Lie-Kolchin theorem, and
by Proposition \ref{split}, we may assume that $G = H = \langle g_1, \dots, g_{n-1} \rangle$.
By Lemma \ref{Lg1}, $L_1':= L_{g_1^{-1}}$ is a nef common eigenvector of $G$.
As in \cite[Proof of Theorem 1.1; \S 2.2]{Z-Tits}, we may construct a quasi-nef sequence
$L_1'\dots L_k'$ ($1 \le k \le n-1$) so that
$g^*(L_1'\dots L_k') = (\chi_1'(g) \dots \chi_k'(g)) (L_1'\dots L_k')$
and the homomorphism below
$$\begin{aligned}
\varphi': \, G & \, \to \, (\R^{\oplus n-1}, +) \\
g \, &\mapsto \, (\log \chi_1'(g), \dots, \log \chi_{n-1}'(g))
\end{aligned}$$
satisfies
$$\Ker(\varphi') = N(G), \hskip 1pc \Imm(\varphi') \cong \Z^{\oplus n-1} .$$
By Hypothesis (C) (see \ref{proj}),
the projection $\Imm(\log \chi_1') \cong \Z$ and hence is generated by some $\log \mu:= \log \chi_1'(g_b)$.
As in the proof of Lemma \ref{Lg1}, $d_1(g_s) L_1' = (g_s^{-1})^*L_1'= \chi_1'(g_s^{-1}) L_1'= \mu^t L_1'$
with $t \in \Z_{> 0}$ depending on $g_s$.
Hence $d_1(g_s) = d_1(g_b)^t$.
Replacing $g_i$'s by their powers, we may assume that
$\lambda := d_1(g_i)$ is independent of $i \in \{1, \dots, n-1\}$.

For every $g \in G \setminus N(G)$, since $G \, | \, H^{1,1}(X)$ is commutative now,
we may assume that $L_g$ is a nef common eigenvector of $G$.
If $L_{g_1} \dots L_{g_{n-1}} . L_g \ne 0$, then applying $g^*$
we get
$$\chi_1(g) \dots \chi_{n-1}(g) . d_1(g) = 1.$$
If $L_{g_1} \dots L_{g_{n-1}} . L_g = 0$, then as
in Lemma \ref{newL}, \cite[Lemma 4.4]{DS} implies that $d_1(g) = \chi_i(g)$ for some $i$.
Thus, $d_1(g)$ is either $\chi_i(g)$ or $\prod_{j=1}^{n-1} \chi_j(g^{-1})$,
which is an integer power of $\lambda$, if we express $g$ as a product of
powers of $g_i$, use Lemma \ref{newL}(7) and note that $\chi_i(g_j) = 1$ ($i \ne j$).
This proves Proposition \ref{primelt}.

\begin{setup}
With the notation and assumption of Theorem $\ref{ThA}$,
suppose that $c_2(X) \ne 0$ as an element in $N^{n-2}(X)$, i.e.,
as a linear form on $N^1(X) \times \cdots \times N^1(X)$, $(n-2)$ of them
(cf. \cite[page 265-267]{SW}, \cite[Definition 2.4]{CY3}).
Since $X$ is minimal, Miyaoka's pseudo-effectivity of $c_2(X)$ implies that
$c_2(X) \ge 0$ on $\Nef(X) \times \cdots \times \Nef(X)$, $(n-2)$ of them
(cf. \cite[Theorem 4.1]{SW} and the reference therein).
Hence $c_2(X) > 0$ on the self product of the ample cone
$\Amp(X) \times \cdots \times \Amp(X)$, since $N^1(X)$ is spanned by $\Amp(X)$
which is an open cone.

Since $c_2(X) . \Amp(X)$ is a nonzero cone of $N^3(X)$
and is stable under the natural action of $G$, and
since $G \, | \, N^1(X)$ is solvable and $Z$-connected,
there are nef (and indeed ample) divisors $P(t)$ such that
$$c_2(X) . M_1 := \lim_{t \to \infty} c_2(X) . P(t)$$
is nonzero in $N^3(X)$ and a common $G$-eigenvector
(cf. \cite[Theorem 1.1]{KOZ});
note that $M_1 \in N^1(X)$ may {\it not} be nef since the cone $c_2(X) . \Nef(X)$ may not be closed in $N^3(X)$.
By the same reasoning, we can construct $c_2(X) . M_1 . M_2$.
Continuing the process, we get:
\end{setup}

\begin{lemma}\label{c2}
With the notation and assumption of Theorem $\ref{ThA}$,
suppose that $c_2(X) \ne 0$ in $N^2(X)$.
Then we have:
\begin{itemize}
\item[(1)]
There is a sequence $0 \ne c_2(X) . M_1 \dots M_k \in N^{k+2}(X)$ $(1 \le k \le n-3)$
$($called a pseudo-effective sequence$)$
which is positive on the self product of the ample cone $\Amp(X) \times \cdots \times \Amp(X)$ $(n - k - 2$ of them$)$,
such that
$$g^*(c_2(X) . M_1 \dots M_k) = \chi_1'(g) \dots \chi_k'(g) (c_2(X) . M_1 \dots M_k)$$
for all $k$ with characters $\chi_i' : G \to (\R_{> 0}, \times)$.
\item[(2)]
In particular, $C := c_2(X) . M_1 \dots M_{n-3}$ is a nonzero element in the closed cone
$\NE(X)$ of effective $1$-cycles $($which is dual to the nef cone $\Nef(X))$.
\end{itemize}
\end{lemma}

\begin{setup}
{\bf Proof of Theorem \ref{ThA}}
\end{setup}

The proof will almost fill up the rest of the paper.
We may and will freely replace
$G$ by its finite-index subgroups.
We may assume that $G = N(G) \, H$ with $H = \langle g_1, \dots, g_{n-1} \rangle$
so that $H \, | \, \NS_{\C}(X) \cong \Z^{\oplus n-1}$
as in Theorem \ref{ThB} and satisfies the four
assertions there but with $H^{1,1}(X)$ replaced by $\NS_{\C}(X)$ (cf. Remark \ref{rThB}).
We use the notation in the proof of Theorem \ref{ThB} and let
$$A_H := L_{g_1} + \dots + L_{g_{n-1}} + L_{g_1^{-1}}$$
be the nef and big divisor,
where $L_{g_j^{\pm}} \in \Nef(X)$ can be chosen to be common eigenvectors of $H$
since $H \, | \, \NS_{\R}(X)$ is commutative.

\begin{lemma}\label{Hfix}
\begin{itemize}
\item[(1)]
Let $D \in N^s(X)$ $(0 < s < n)$
such that $h^*D = D$
for all $h \in H$. Then $D . A_H^{n-s} = 0$.
\item[(2)]
In particular, for the Chern classes $c_i(X)$ $(i = 1, 2)$,
we have $c_i(X) . A_H^{n-i} = 0$.
\item[(3)]
Hence $K_X \sim_{\Q} 0$, i.e.,
a positive multiple of $K_X$ is linearly equivalent to zero.
\end{itemize}
\end{lemma}

\begin{proof}
Take $M := L_1^{i_1} \cdots L_{n}^{i_n}$ with $\sum_{k=1}^n i_k = n-s$.
For $h \in H$, we have $h^*M = e(h) M$ with $e(h) = \chi_1(h)^{i_1} \cdots \chi_n(h)^{i_n}$.
Since $\R \ni M . D = h^*M . h^*D = e(h) M . D$ and since $A_H^{n-s}$ is a combination of such $M$,
it suffices to show that $e(h) \ne 1$ for some $h \in H$ (so that $M . D = 0$).
Suppose the contrary that $e(h) = 1$ for all $h \in H$.
Taking $\log$ and using $\chi_1 \dots \chi_n = 1$, we have
$(i_1-i_n) \log \chi_1 + \cdots + (i_{n-1}-i_n) \log \chi_{n-1} = 0$ on $H$.
Since the image of the homomorphism $\varphi = (\log \chi_1, \dots, \log \chi_{n-1})$ is a spanning lattice
in $\R^{\oplus n-1}$
of rank $n-1$ (cf. \ref{phi}),
this happens only when $i_1 - i_n = \cdots = i_{n-1} - i_n = 0$.
Thus $n-1 \ge n-s = \sum_{k=1}^n i_k = n i_1$, so $i_1 = 0$ and hence $s = n$.
This is absurd.

Since $X$ is minimal and hence $K_X$ is nef, the vanishing of $K_X . A_H^{n-1}$
and \cite[Lemma 2.2]{nz2} imply that $K_X \equiv 0$ (numerically).
Hence $K_X \sim_{\Q} 0$ (cf. \cite[Theorem 8.2]{Ka85}).
This proves Lemma \ref{Hfix}.
\end{proof}

\begin{lemma}\label{ampleA}
Theorem $\ref{ThA}$ is true when $c_2(X) = 0$ in $N^2(X)$.
\end{lemma}

We prove the lemma. Set $A := A_H$.
The vanishing of $c_2(X)$ in $N^2(X)$
implies the vanishing of the orbifold second Chern class of $X$ (cf. \cite[Proposition 1.1]{SW}).
This and the vanishing of $c_1(X)$ in Lemma \ref{Hfix} imply the existence of
a finite surjective morphism $T' \to X$ from an abelian variety $T'$, based on a deep result of S. T. Yau
(cf. \cite{Be}, \cite[Theorem 7.6]{CZ}).
This is the place we need the singularities of $X$ to be of quotient type
(cf. also Remark \ref{rThA}).

Since $K_{T'} \sim 0 \sim_{\Q} K_X$,
the map $T' \to X$ is \'etale in codimenion one.
Let $T \to X$ be the Galois cover corresponding to the unique maximal lattice $L$
in $\pi_1(X \setminus \Sing X)$ so that $T$ is an abelian variety.
Then $X = T/F$ with
$$F = \pi_1(X \setminus \Sing X)/L = \Gal(T/X)$$
and there is an exact sequence
$$1 \to F \to \widetilde{G} \overset{\gamma}\to G \to 1$$
where $\widetilde{G}$ (the lifting of the original $G$) acts faithfully on $T$
(cf. \cite[\S 3, especially Proof of Prop 3]{Be} applied to
\'etale-in-codimension-one covers, also \cite[Prop 3.5]{nz2}).

By \cite[Lemma 2.6]{JDG} or \cite[Lemma A.8]{NZ}, $N(\widetilde{G}) = \gamma^{-1}(N(G))$.
Hence
$$\widetilde{G}/N(\widetilde{G}) \cong G/N(G) \cong \Z^{\oplus n-1}, \,\,\,\,\,
r(\widetilde{G}) = r(G) = n-1 .$$
Since $N(\widetilde{G}) \, | \, \NS_{\C}(X)$ is virtually unipotent (cf. Proposition \ref{split}),
we may assume that $\widetilde{G} \, | \, \NS_{\C}(X)$ is solvable and $Z$-connected after
$G$ is replaced by its finite-index subgroup.

If Theorem \ref{ThA}(1) is not true then $N(\widetilde{G}) \, | \, \NS_{\R}(T)$ is also an infinite group,
thus \cite[Proof of Theorem 1.1(3), page 2338]{CY3}
shows the existence of a $\widetilde{G}$-equivariant fibration
$T \to T/B$ with $0 < B < T$ a subtorus of $T$ fixed by the unipotent elements of $\widetilde{G} \, | \, H^{1,1}(T)$;
this leads to rank $r(G) = r(\widetilde{G}) \le n - 2$
(cf. \cite[Proof of Lemma 2.10]{Z-Tits}), contradicting the assumption $r(G) = n-1$.
Hence the assertion (1) of Theorem \ref{ThA} is true.
(2) follows from (1) and Lemma \ref{modfin}.

For the assertion (3), take an ample divisor (or a K\"ahler class for the purpose of later
Theorem \ref{ThC}) $M'$ on $X$.
Then $M := \sum h_t^* M'$, where $h_t$ runs in the
finite group $N(G) \, | \, \NS_{\R}(X)$,
is an ample divisor (or a K\"ahler class) and stable under the action of $N(G)$.
Hence $N(G) \le \Aut_M(X) := \{f \in \Aut(X) \, | \,
f^*M = M\}$, where $|\Aut_M(X) : \Aut_0(X)| < \infty$
by \cite[Proposition 2.2]{Li} or \cite[Theorem 4.8]{Fu}.
Thus for
$$N_0 := N(G) \cap \Aut_0(X)$$
we have
$$N(G) / N_0 \cong (\Aut_0(X) . N(G)) / \Aut_0(X) \le \Aut_M(X) / \Aut_0(X)$$
where the latter is a finite group.
If $N_0$ is finite, then so is $N(G)$ and the first case of the assertion (3) is true (cf. Lemma \ref{modfin}).

Suppose that $N_0$ is infinite. We shall show that the second case of the assertion (3) occurs.
First, $\Aut_0(X) \ne 1$.
If the linear part of $\Aut_0(X)$ is non-trivial, then $X$ is ruled (and hence uniruled)
(cf. \cite[Proposition 5.10]{Fu}), contradicting the assumption that $X$ is a minimal variety with only canonical
singularities and hence non-uniruled by the well-known Miyaoka-Mori uniruledness criterion.

Therefore, the linear part $\Aut_0(X)$ is trivial. Then $\Aut_0(X)$ is a complex torus
and (cf. \cite[Theorem 3.12]{Li} or \cite[Theorem 5.5]{Fu})
$$1 \le \dim \Aut_0(X) \le \dim \Alb(X) = q(X)$$
where $\Alb(X)$ is the Albanese variety of $X$ and $q(X) = h^1(X, \OO_X)$ is the
irregularity of $X$.
By \cite[Lemma 2.13]{Z-Tits} applied to a resolution of $X$
and using \cite[Lemma 8.1]{Ka85}, the albanese map $\alb_X : X \to \Alb(X)$ is a birational and surjective morphism.

\begin{claim}\label{per}
$X$ has no positive-dimensional subvariety $Y'$ which is $G$-periodic
$($i.e., $Y'$ is stabilized by a finite-index subgroup of $G)$.
\end{claim}

\begin{proof}
We prove the claim.
We have already proved the finiteness of $N(\widetilde{G}) \, | \, \NS_{\C}(T)$.
Replacing $\widetilde{G}$ by its finite-index subgroup and by Lemma \ref{modfin},
we may assume that
$\widetilde{G} \, | \, \NS_{\C}(T)$
and even $\widetilde{G}/(\widetilde{G} \cap \Aut_0(T))$
are free abelian groups of rank $n-1$.

Suppose a positive-dimensional subvariety $Y' \subset X$ is $G$-periodic.
Then a subvariety $Y \subset T$ (dominating $Y'$) is $\widetilde{G}$-periodic
and is even stabilized by $\widetilde{G}$ after this group is replaced by its finite-index subgroup.
By the proof of \cite[Lemma 2.11]{per}, there is a $\widetilde{G}$-equivariant
homomorphism $T \to T/B$ with $\dim (T/B) \in \{1, \dots, n-1\}$.
Thus the rank $r(\widetilde{G}) \le \dim T - 2 = n-2$ (cf. \cite[Lemma 2.10]{Z-Tits}).
This contradicts the fact that $r(\widetilde{G}) = r(G) = n - 1$.
This proves Claim \ref{per}.
\end{proof}

\par \vskip 1pc
We return to the proof of Lemma \ref{ampleA}.

The action of (the original group) $G$ on $X$ induces an action of $G$ on
the Albanese variety $\Alb(X)$ so that $\alb_X$ is $G$-equivariant,
by the universal property of the albanese variety.
If $\alb_X$ is not an isomorphism, then its exceptional locus (where the map is not isomorphic)
is $G$-periodic and positive-dimensional by Zariski's main theorem,
contradicting Claim \ref{per}.
Thus $\alb_X$ is an isomorphism, and hence $X$ is an abelian variety.
Since $r(G) = n-1$, \cite[Lemma 2.14]{Z-Tits}
implies that the Zariski-closure $\overline{N_0}$ of $N_0$ in the translation group $\Aut_0(X) \cong X$
acts (as translations) on the torus $X$ with a Zariski-dense open orbit
and is hence equal to $\Aut_0(X)$. So the second case of the assertion (3) occurs.

For Theorem \ref{ThA}(4), since $F \lhd \widetilde{G}$,
the fixed locus
$$T^F = \{t \in T \, | \, f(t) = t \,\,\,\, \text{for some} \,\,\,\, \id \neq f \in F\}$$
is $\widetilde{G}$-stable. If this locus is positive-dimensional, then
its image in $X$ is $G$-periodic, contradicting Claim \ref{per}.
If $X$ is not an abelian variety yet, then by the proved assertion (3),
$N(\widetilde{G}) = \gamma^{-1}(N(G))$ is finite.
Applying Lemma \ref{modfin} to
$$(\widetilde{G}/F)/(N(\widetilde{G})/F) \cong G/N(G) \cong \Z^{\oplus n-1}$$
there is a rank $n-1$, free abelian, finite-index subgroup $\widetilde{G_1}/F$ of $\widetilde{G}/F = G$.
Applying Lemma \ref{modfin} to $\widetilde{G_1}/F$,
our $\gamma : \widetilde{G} \to G$ maps a rank $n-1$, free abelian, finite-index subgroup $G_1$ of $\widetilde{G}$
isomorphically onto a subgroup of $\widetilde{G_1}/F \le G$.
This proves Theorem \ref{ThA}(4), Lemma \ref{ampleA} and
the whole of Theorem \ref{ThA}, provided that
$c_2(X) = 0$ in $N^2(X)$.

\par \vskip 1pc
Now Theorem \ref{ThA} follows from the following:

\begin{lemma}\label{c2=0}
$c_2(X) = 0$ in $N^2(X)$.
\end{lemma}

We prove the lemma.
Suppose the contrary that $c_2(X) \ne 0$ in $N^2(X)$.
We shall show that this contradicts the minimality assumption on the pair $(X, G)$.
There is an extremal birational contraction $\tau = \tau_H: X \to X_H$
corresponding to the
rational polyhedral face
$$F_H := \{\alpha \in \NE(X) \, | \, A_H . \alpha = 0\}$$
of the closed cone of effective
curves $\NE(X)$ (cf. \cite[Proof of Theorem 3.9.1]{BCHM}),
so that $\tau(C)$ is a point for a curve $C$ on $X$ if and only if the class $[C] \in F_H$,
and $A_H = \tau^* A_H'$ (resp. $L_{g_i} = \tau^*L_{g_i}'$) for an ample $\R$-divisor $A_H'$
(resp. a nef $\R$-divisor $L_{g_i'}$);
by \cite[Proof of Claim 2.11]{CY3}, $X_H$ has only canonical singularities and
$K_{X_H} \sim 0$, because the same assertions are true for $X$.
Since $A_H$ is the sum of nef $H$-eigenvectors, $F_H$ is $H$-stable.
So the extremal contraction $\tau$ is $H$-equivariant.

\begin{claim}\label{Hper}
$X_H$ has no $H$-periodic subvariety $D$ of dimension $s$ in $\{1, \dots, n-1\}$.
Hence
every $H$-periodic subvariety of $X$ $($especially
the exceptional locus of $\tau_H)$ is
contracted to a point(s) by $\tau_H$.
\end{claim}

\begin{proof}
We prove the claim.
Since $\tau : X \to X_H$ is $H$-equivalent and $L_{g_i} = \tau^*L_{g_i}'$,
our $L_{g_i}'$ and $L_{g_i}$ give rise to the
same character $\chi_i$ on $H$.
Now the proof of Lemma \ref{Hfix}
shows that $(A_H')^{s} . D = 0$.
Since $A_H'$ is ample, this contradicts Nakai's ampleness criterion
(generalized to $\R$-divisors by Campana and Peternell).
This proves Claim \ref{Hper}.
\end{proof}

\par \vskip 1pc
We return back to the proof of Lemma \ref{c2=0}

Take $u \in N(G)$ and set $H_u := u^{-1}Hu$.
Then $G = N(G) \, H_u$ and it satisfies the four assertions of Theorem \ref{ThB} (with $H$ replaced by $H_u$).
Set $g_i' := u^{-1} g_i u$ so that $H_u = \langle g_1', \dots, g_{n-1}' \rangle$.
Since $(g_i')^* (u^*L_{g_i}) = u^* g_i^* (u^{-1})^* (u^*L_{g_i}) = d_1(g_i) u^*L_{g_i}$,
we may take $L_{g_i'} = u^*L_{g_i}$ and similarly $L_{g_i^{-1}} = u^* L_{g_1^{-1}}$ (cf. Lemma \ref{Lg1})
which are all nef common $H_u$-eigenvectors.
Set
$$A_{H_u} = L_{g_1'} + \dots + L_{g_{n-1}'} + L_{{g_1'}^{-1}}, \hskip 1pc
u^* F_H := \{u^*(\alpha) \, | \, \alpha \in F_H\} .$$
Then $A_{H_u} = u^*A_H$,  and
$$u^*F_H = F_{H_u} := \{\alpha \in \NE(X) \, | \, A_{H_u} . \alpha = 0\}.
$$
Since each $u^*F_H$ is spanned by finitely many extremal rays in the cone $\NE(X)$,
there are finitely many $u_i \in N(G)$ ($i = 1, \dots, t$) such that
$$\begin{aligned}
F_G :&= \cap_{g \in G} \, g^*F_H = \cap_{u \in N(G)} \, u^*F_H
= \cap_{i=1}^t \, u_i^*F_H  \\
&= \{\alpha \in \NE(X) \, | \, A_{H_{u_i}}  . \alpha = 0, \, 1 \le i \le t\}
= \{\alpha \in \NE(X) \, | \, \sum_{i=1}^t A_{H_{u_i}}  . \alpha = 0\}
\end{aligned}
$$
where the second equality is true because $G = N(G) \, H$ and $A_H$ is the sum
of nef $H$-eigenvectors.

\begin{claim}\label{FG}
$F_G = 0$.
\end{claim}

\begin{proof}
We prove the claim.
Suppose the contrary that $F_G \ne 0$.
Note that $F_G$ is $G$-stable.
As argued before Claim \ref{Hper}, there is a non-isomorphic $G$-equivariant
birational extremal contraction $\tau_G : X \to X_G$ so that $\tau_G(C)$ is a point
for a curve $C$ on $X$ if and only if the class $[C] \in F_G$;
$$A_G := \sum_{i=1}^t A_{H_{u_i}} = \tau_G^* A_G', \hskip 1pc
L_{g_i'} = \tau_G^*L_{g_i'}'$$
for some ample divisor $A_G'$ and nef divisor $L_{g_i'}'$ on $X_G$
($L_{g_i'}$, w.r.t. every given $u$, being defined preceding this claim)
and $X_G$ has at worst canonical singularities.
The same proof of Claim \ref{Hper} (applied to all $H_u$) and the ampleness of $A_G'$
show that $X_G$ has no positive-dimensional $G$-periodic subvariety
and hence $\dim \Sing X_G \le 0$.
This contradicts the minimality of the pair $(X, G)$.
Therefore, $F_G = 0$ and Claim \ref{FG} is true.
\end{proof}

\par \vskip 1pc
We return to the proof of Lemma \ref{c2=0}.
Let $\sigma : Z \to X$ be some Hironaka's blowup with centre in the exceptional
locus of $\tau = \tau_H : X \to X_H = :Y$, such that
$-E$ is $\tau \sigma$-relatively ample for some effective $\tau \sigma$-exceptional divisor
$E$ on $Z$.
Since $A = \tau^*A'$ for some ample divisor $A'$ on $Y$
(with $A := A_H$ and $A' := A_H'$),
$\sigma^* A - E$ is an ample divisor on $Z$ if $A$ is replaced by a positive multiple
(cf. \cite[Proposition 1.45]{KM}).
By Claim \ref{Hper} and the construction of the blowup $\sigma$,
we have $\dim \tau\sigma(E) \le 0$, so $\sigma^*A . E = (\tau \sigma)^*A' . E = 0$.
The non-vanishing of $c_2(X)$ in $N^2(X)$ and Lemma \ref{c2} imply the existence
of a nonzero effective $1$-cycle
$C := c_2(X) . M_1 \dots M_{n-3}$.
Since ampleness is an open condition, replacing $M_i$ by a small positive multiple, we may write
$$\sigma^*A - E = \sigma^*M_i + P_i$$
for some ample $\R$-divisor $P_i$.
Since the difference $\sigma_*c_2(Z) - c_2(X)$ (as $(n-2)$-cycles)
lies in the centre of the blowup of $\sigma : Z \to X$ (i.e., in the exceptional locus
of $\tau : X \to Y$ by the construction of $\sigma$) and hence is contracted by $\tau$
to a finite subset of $Y$, and since $A = \tau^*A'$, we have
$\sigma_* c_2(Z) . A = c_2(X) . A$ as $(n-3)$-cycles on $X$.

Consider the $1$-cycle
$\ell(i_1, \dots, i_s) := \sigma_*(c_2(Z) . P_{i_1} \dots P_{i_s}) A^{n-s-3}$ on $X$.
Since $\sigma : Z \to X$ is isomorphic
outside the exceptional locus of $\tau : X \to Y$ by the construction
and the intersection of $c_2(X)$ with $\Nef(X) \times \cdots \times \Nef(X)$ ($n-3$ of them)
lies in $\NE(X)$ (Miyaoka's pseudo-effectivity),
our $\ell(i_1, \dots, i_s) = \ell_1(i_1, \dots, i_s) + \ell_2(i_1, \dots, i_s)$ with $\ell_1$ effective and $\ell_2$ supported on
the exceptional locus of $\tau$ and hence $\ell_2(i_1, \dots, i_s) . A = \ell_2(i_1, \dots, i_s) . \tau^*A' = 0$
(cf. Claim \ref{Hper}).

Now Miyaoka's pseudo-effectivity of $c_2$ for minimal variety, the construction of
the pseudo-effective sequence in Lemma \ref{c2},
$\sigma_* c_2(Z) . A = c_2(X) . A$, the projection formula,
$\sigma^*A . E = 0$ and Lemma \ref{Hfix} imply:
$$\begin{aligned}
0 &\le c_2(X) . M_1 \dots M_{n-3} . A
= \sigma_*c_2(Z) . M_1 \dots M_{n-3} . A
= c_2(Z) . \sigma^*M_1 \dots \sigma^*M_{n-3} . \sigma^*A \\
&= c_2(Z) . (\sigma^*A - P_1 - E) \dots (\sigma^*A - P_{n-3} - E) . \sigma^*A \\
&= c_2(Z) . (\sigma^*A - P_1) \dots (\sigma^*A - P_{n-3}) . \sigma^*A \\
&= c_2(Z) . (\sigma^*A)^{n-2}
- \sum_{i_1, \dots, i_s} (c_2(Z) . P_{i_1} \dots P_{i_s}) (\sigma^*A)^{n-s-2} \\
&= \sigma_*c_2(Z) . A^{n-2}
- \sum_{i_1, \dots, i_s} \sigma_*(c_2(Z) . P_{i_1} \dots P_{i_s}) A^{n-s-2} \\
&= c_2(X) . A^{n-2}
- \sum_{i_1, \dots, i_s} (\ell_1(i_1, \dots, i_s) + \ell_2(i_1, \dots, i_s)) A \\
&= c_2(X) . A^{n-2}
- \sum_{i_1, \dots, i_s} \ell_1(i_1, \dots, i_s). A
\le c_2(X) . A^{n-2} = 0.
\end{aligned}$$
Thus the effective $1$-cycle $C = c_2(X) . M_1 \dots M_{n-3}$
satisfies $C . A_H = 0$ (with $A = A_H$).
Since $C$ is a $G$-eigenvector (cf. Lemma \ref{c2}), we then have $C . u_i^*A_H = 0$
(with $u_i^*A_H = A_{H_{u_i}}$) for all $i$.
Hence $C \in F_G = 0$. This contradicts $C \ne 0$ (cf. Lemma \ref{c2}).
We have completed the proof of Lemma \ref{c2=0} and also Theorem \ref{ThA}.

\begin{setup}{\bf Proof of Theorem \ref{ThC}}
\end{setup}

We will use the argument till \ref{Pfprimelt}.
We may and will freely replace
$G$ by its finite-index subgroups.
We may assume that $G = N(G) \, H$ with $H = \langle g_1, \dots, g_{n-1} \rangle$
so that $H \, | \, H^{1,1}(X) \cong \Z^{\oplus n-1}$
as in Theorem \ref{ThB} and satisfies the four
assertions there.
We use the notation in the proof of Theorem \ref{ThB} and let
$$A_H := L_{g_1} + \dots + L_{g_{n-1}} + L_{g_1^{-1}}$$
be the nef and big class,
where $L_{g_j^{\pm}} \in \overline{K(X)}$ can be chosen to be common eigenvectors of $H$
since $H \, | \, H^{1,1}(X)$ is commutative.

As in Lemma \ref{Hfix}, utilizing \cite[Lemma 2.2]{nz2} (writing the nef and big class $A_H$
as the sum of a K\"ahler class and a positive real current),
$K_X$ equals zero (cohomologously).
Thus there is an \'etale finite Galois covering $\widetilde{X} = T \times S \times Y \to X$ such that
$T$ is a complex torus, $S$ is a product of hyperk\"ahler manifolds $S_i$, $Y$ is a product
of (projective) Calabi-Yau manifolds $Y_j$, and $G$, replaced by its finite-index subgroup,
lifts to some group
$$\widetilde{G} \, \le \, \Aut(T) \times \prod \Aut(S_i) \times \prod \Aut(Y_j)$$
with $\widetilde{G}/\Gal(\widetilde{X}/X) = G$ (cf. \cite[\S 3]{Be}).
As in Theorem \ref{ThA},
$N(\widetilde{G})$ is the preimage of $N(G)$, via the
quotient map $\widetilde{G} \to G$, so that $\widetilde{G}/N(\widetilde{G}) = \Z^{\oplus n-1}$, and hence
the rank $r(\widetilde{G}) = n-1$.
Since the projections of $\widetilde{X}$ to its factors $T$, $S_i$ and $Y_j$
are $\widetilde{G}$-equivariant,
the maximality of $r(\widetilde{G})$ implies that $\widetilde{X}$ equals $T$, $S = S_1$ or $Y = Y_1$,
i.e., is a complex torus, a hyperk\"ahler manifold, or a projective Calabi-Yau manifold (cf. \cite[Lemma 2.10]{Z-Tits}).

Since $X$ is non-algebraic, so is $\widetilde{X}$, and hence $\widetilde{X}$ equals $T$ or $S$.
If $\widetilde{X}$ is hyperk\"ahler, then we reach a contradiction: $2 \le n-1 = r(\widetilde{G}) \le 1$
(cf. \cite[Theorem 4.6]{KOZ}). Thus $\widetilde{X}$ is a complex torus.
Now the argument of Lemma \ref{ampleA} (easier now) implies Theorem \ref{ThC}.

\end{document}